\documentclass[notitlepage]{scrartcl}
\usepackage[shortlabels]{enumitem}
\usepackage{amsmath}
\usepackage{amssymb}
\usepackage{amsthm}
\usepackage{abstract}
\usepackage{xcolor}
\usepackage{comment}
\usepackage{mathtools}
\usepackage{graphicx}
\usepackage{caption}
\usepackage{subcaption}
\usepackage{float}
\usepackage{hyperref}
\usepackage{cleveref}
\hypersetup{
    colorlinks=true,
    linkcolor=blue,
    filecolor=magenta,      
    urlcolor=cyan
    }
\makeatletter
\newcommand*{\barfix}[2][.175ex]{%
  \mathpalette{\@barfix{#1}}{#2}%
}
\newcommand*{\@barfix}[3]{%
  \vbox{%
    \kern#1\relax
    \hbox{$#2#3\m@th$}%
  }%
}
\makeatother

\newtheorem{theorem}{Theorem}

\newtheorem{lemma}[theorem]{Lemma}

\newtheorem{claim}[theorem]{Claim}

\newtheorem{question}[theorem]{Question}
\newcommand{\footremember}[2]{%
    \footnote{#2}
    \newcounter{#1}
    \setcounter{#1}{\value{footnote}}%
}
\newcommand{\footrecall}[1]{%
    \footnotemark[\value{#1}]%
} 

\usepackage[margin=1in]{geometry} 
\title{\vspace{-1.5cm}Large matchings and nearly spanning, nearly regular subgraphs of random subgraphs}  
\author{%
Sahar Diskin \footremember{alley}{\scriptsize{School of Mathematical Sciences, Tel Aviv University, Tel Aviv 6997801, Israel. Emails: sahardiskin@mail.tau.ac.il, krivelev@tauex.tau.ac.il.}}%
\and Joshua Erde \footremember{trailer}{\scriptsize{Institute of Discrete Mathematics, Graz University of Technology, Steyrergasse 30, 8010 Graz, Austria. Emails: erde@math.tugraz.at, kang@math.tugraz.at.}}%
\and Mihyun Kang \footrecall{trailer}%
\and Michael Krivelevich \footrecall{alley}%
}
\date{}
\begin{document}
\maketitle

\vspace{-1cm}\begin{abstract}
Given a graph $G$ and $p\in [0,1]$, the random subgraph $G_p$ is obtained by retaining each edge of $G$ independently with probability $p$. We show that for every $\epsilon>0$, there exists a constant $C>0$ such that the following holds. Let $d\ge C$ be an integer, let $G$ be a $d$-regular graph and let $p\ge \frac{C}{d}$. Then, with probability tending to one as $|V(G)|$ tends to infinity, there exists a matching in $G_p$ covering at least $(1-\epsilon)|V(G)|$ vertices. 

We further show that for a wide family of $d$-regular graphs $G$, which includes the $d$-dimensional hypercube, for any $p\ge \frac{\log^5d}{d}$ with probability tending to one as $d$ tends to infinity, $G_p$ contains an induced subgraph on at least $(1-o(1))|V(G)|$ vertices, whose degrees are tightly concentrated around the expected average degree $dp$.
\end{abstract}

\section{Introduction}
A classical result of Erd\H{o}s and R\'enyi \cite{ER66} states that $p=\frac{\log n}{n}$ is the threshold for the existence of a \emph{perfect} matching (that is, a matching covering all but at most one vertex) in $G(n,p)$\footnote{In fact, Erd\H{o}s and R\'enyi worked in the closely related \emph{uniform} random graph model $G(n,m)$.}, which also coincides with the connectivity threshold (see also \cite{BT85} for a hitting time result). Below this threshold, it is not hard to show that with probability tending to one as $n$ tends to infinity a fixed proportion of the vertices are isolated and will not be covered in any matching. On the other hand, it follows from a celebrated result of Karp and Sipser \cite{KS81} from 1981 that,
when $p=\frac{C}{n}$ for a large enough constant $C$, $G(n,p)$ contains a matching on $(1-o_C(1))n$ vertices with probability tending to one as $n$ tends to infinity. Subsequent work by Frieze \cite{F86} gave a precise estimate of the asymptotic proportion of vertices which are not covered by a largest matching in this regime.

The binomial random graph $G(n,p)$ is an instance of the model of \textit{bond percolation}. Given a host graph $G$ and a probability $p\in [0,1]$, we form the random subgraph $G_p \subseteq G$ by retaining each edge of $G$ independently with probability $p$ (indeed, $G(n,p)$ is equivalent to performing bond percolation on the complete graph $K_n$). In a qualitative sense, our first main result extends the result of Karp and Sipser \cite{KS81} to any $d$-regular graph.
\begin{theorem}\label{th: matching}
For every $\epsilon>0$, there exists a constant $C>0$ such that the following holds. Let $d\in \mathbb{N}, d\ge C$, let $G$ be a $d$-regular graph on $n$ vertices, and let $p\ge \frac{C}{d}$. Then, with probability tending to one as $n$ tends to infinity, there exists a matching in $G_p$ covering at least $(1-\epsilon)n$ vertices.
\end{theorem}

In addition to the typical existence of large matchings in percolated $d$-regular graphs, we also explore typical structural properties of the hypercube under percolation. %(Sahar) : I have no idea what to write here, actually. TBH, these two results are unrelated

In the setting of $G(n,p)$, finding a large $k$-regular subgraph when $p=\frac{C}{n}$ has been extensively studied (see, for example, \cite{KSW11, MMP} and references therein). A natural variant is to try and find a \textit{nearly regular}, nearly spanning subgraph (see \cite{AKS08} for an extremal variant of this question). We consider this question in the setting of random subgraphs of the hypercube. Recall that the $d$-dimensional binary hypercube is the graph whose vertex set is $\{0,1\}^d$, and where two vertices are connected if and only if their Hamming distance is one. In our second main result, we establish the typical existence of a nearly regular, nearly spanning induced subgraph of $Q^d_p$, whose degrees are tightly concentrated around the expected degree $dp$.
\begin{theorem}\label{th: induced}
For every $\epsilon>0$, there exists $d_0\in \mathbb{N}$ such that the following holds for every integer $d\ge d_0$. Let $p\ge \frac{\log^5 d}{d}$. Then, with probability tending to one as $d$ tends to infinity there exists an induced subgraph $H\subseteq Q^d_p$, such that $|V(H)|\ge (1-\epsilon)2^d$ and for every $v \in V(H)$, $\left|1-\frac{d_H(v)}{dp}\right|\le \epsilon$.
\end{theorem}
We note that the proofs of the theorems are quite short. Further, let us remark that we have not tried to optimise the polylogarithmic dependence of $p$ on $d$ in Theorem \ref{th: induced}. Finally, we note that we present the proof for $Q^d$, but, as will be expanded upon in the discussion section, Theorem \ref{th: induced} holds for a fairly wide family of graphs (see Theorem \ref{th: general}). 

In Section \ref{s: proofs} we prove Theorems \ref{th: matching} and \ref{th: induced}. In Section \ref{s: discussion} we discuss our results and avenues for future research.

\section{Proofs of Theorems \ref{th: matching} and \ref{th: induced}}\label{s: proofs}
Given a graph $G$ and $v\in V(G)$, we denote by $d_G(v)$ the degree of $v$ in $G$. For $k\in\mathbb{N}$, we denote by $N_G^k(v)$ the set of vertices at distance exactly $k$ from $v$ in $G$, where $N_G(v)\coloneqq N_G^1(v)$. When the graph $G$ is clear from the context, we will omit the subscripts. Further, given $S\subseteq V(G)$, we denote by $G[S]$ the subgraph of $G$ induced by the vertices of $S$. Given $x,y,z \in \mathbb{R}$ we write $x = y \pm z$ as shorthand for $x \in [y-z, y+z]$. Throughout the paper, all logarithms are in the natural base.

We will make use of a typical Chernoff-type tail bound on the binomial distribution (see, for example, Appendix A in \cite{AS16}).
\begin{lemma}\label{l: chernoff}
Let $d\in \mathbb{N}$, let $p\in [0,1]$, and let $X\sim \mathrm{Bin}(d,p)$. Then for any $0<t\le \frac{dp}{2}$, 
\begin{align*}
    &\mathbb{P}\left[X\neq dp\pm t\right]\le 2\exp\left\{-\frac{t^2}{3dp}\right\}.
\end{align*}
\end{lemma}

We will also utilise a variant of the well-known Azuma-Hoeffding inequality (see, for example, Chapter 7 in \cite{AS16}),
\begin{lemma}\label{l: azuma}
Let $m\in \mathbb{N}$ and let $p\in [0,1]$. Let $X = (X_1,X_2,\ldots, X_m)$ be a random vector with range $\Lambda = \{0,1\}^m$ with each $X_{i}$ distributed according to independent Bernoulli$(p)$. Let $f:\Lambda\to\mathbb{R}$ be a function such that there exists $K \in \mathbb{R}$ such that for every $x,x' \in \Lambda$ which differ only in one coordinate, $|f(x)-f(x')|\le K$. Then, for every $t\ge 0$,
\begin{align*}
    \mathbb{P}\left[\big|f(X)-\mathbb{E}\left[f(X)\right]\big|\ge t\right]\le 2\exp\left\{-\frac{t^2}{2K^2mp}\right\}.
\end{align*}
\end{lemma}

\subsection{Proof of Theorem \ref{th: matching}}
Fix $\epsilon>0$. Let $\delta\coloneqq \delta(\epsilon)>0$ and $C \coloneq C(\delta,\epsilon)=C(\epsilon) >0$ be constants satisfying that $C\exp\left\{-\frac{\delta^2C}{16}\right\}\le \frac{\epsilon}{4}$ and $\frac{C-\epsilon}{(1+\delta)C}\ge 1-\epsilon$. Note that by monotonicity, we may assume $p= \frac{C}{d}$. Let $V_0$ be the set of vertices of degree at least $(1+\delta)C$, and let $E_0$ be the set of edges touching $V_0$, that is,
\begin{align*}
    V_0\coloneqq \{v\in V(G)\colon d_{G_p}(v)\ge (1+\delta)C\} \quad \text{ and }  \quad E_0\coloneqq \left\{ uv \in E\left(G_p\right) \colon \{u,v\} \cap V_0 \neq \emptyset\right\}.
\end{align*}

For any $e\in E(G)$ the probability that $e\in E(G_p)$ \textit{and} at least one of its endpoints is in $V_0$ is at most $p\cdot 2\mathbb{P}\left(\text{Bin}\left(d,\frac{C}{d}\right)\ge (1+\delta)C-1\right)$, which is bounded above by $\frac{C}{d}\cdot 4\exp\left\{-\frac{\delta^2C}{4}\right\}$ by Lemma \ref{l: chernoff} (where we assumed that $\delta C\ge 9$). Hence, using $|E(G)|=\frac{dn}{2}$, we have $\mathbb{E}\left[|E_0|\right]\le 2C\exp\left\{-\frac{\delta^2C}{4}\right\}n \leq \frac{\varepsilon}{8}n$. Now, note that adding/removing any edge can change $|V_0|$ by at most two, and thus $|E_0|$ by at most $2(1+\delta)C$. Hence, by Lemma \ref{l: azuma} (with $t=\frac{\epsilon n}{8}$, $K=2(1+\delta)C$, and $m=\frac{nd}{2}$), and using $\exp\left\{-\frac{\delta^2C}{4}\right\}\le \frac{\epsilon}{16 C}$ we obtain
\begin{align*}
    \mathbb{P}\left[|E_0|\ge \frac{\epsilon}{4}n\right]\le 2\exp\left\{-\frac{(\epsilon n/8)^2}{2\cdot 4(1+\delta)^2C^2\cdot nd/2\cdot C/d}\right\}=\exp\left\{-\Omega(n)\right\}=o_n(1).
\end{align*}
By similar (and simpler) arguments, with probability $1-o_n(1)$, $|E(G_p)|\ge \frac{ndp}{2}-\frac{\epsilon n}{4} = \frac{Cn}{2} - \frac{\epsilon}{4} n$.  

Let $H$ be the subgraph of $G_p$ induced by $V\setminus V_0$. Note that for every $v\in V(G)$, we have that $d_{H}(v)<(1+\delta)C$. Hence, by Vizing's theorem \cite{V64}, there exists a proper colouring of $H$ with $(1+\delta)C$ colours. Therefore, with probability $1-o_n(1)$ there is a matching in $H$ (and thus in $G_p$) covering at least
\begin{align*}
    \frac{2 |E(H)|}{(1+\delta)C}=\frac{2(|E(G_p)|-|E_0|)}{(1+\delta)C}\ge \frac{(C - \varepsilon)}{(1+\delta)C}n\ge (1-\epsilon) n
\end{align*}
vertices, as required. \qed

\subsection{Finding a nearly regular, nearly spanning subgraph}
In this section, we prove Theorem \ref{th: induced}. For ease of presentation, we write $G \coloneqq Q^d$, so that $G_p = Q^d_p$. Throughout the section, we assume $d\in\mathbb{N}$ is sufficiently large, and all asymptotic notation in this section will be with respect to the parameter $d$. Further, throughout this section, we let $\delta\coloneqq \delta(d)$ be a function tending to $0$ arbitrarily slowly as $d$ tends to infinity. We recall our assumption that $dp\ge \log^5d$, and in particular, $dp=\omega(1)$ (we will use this fact at various points in the proof).

We will analyse the following `pruning' process on $G_p$. For each $t \in \mathbb{N}$ let $$\delta_t \coloneqq \frac{t\cdot \delta}{\lfloor\log d\rfloor}.$$ Let $\tau\coloneqq \lfloor\log d\rfloor$ so that $\delta_{\tau}=\delta$. Let $H_1\coloneqq G_p$, and let
\begin{align}\label{e:A1}
    A_1\coloneqq \left\{v\in V(H_1) \colon d_{H_1}(v)\neq (1 \pm \delta_1)dp\right\}.
\end{align}
We then proceed as follows: At the $t$-th iteration, for $t\in [2,\tau]\cap\mathbb{N}$, we let $H_t\coloneqq G_p\left[V(G)\setminus \bigcup_{i=1}^{t-1}A_i\right]$, and let 
\begin{align*}
    A_t\coloneqq \left\{v\in V(H_{t})\colon d_{H_{t}}(v)< (1-\delta_t)dp\right\}.
\end{align*}
We run this process until $t= \tau$, and let 
\begin{align*}
H\coloneqq H_{\tau}\quad \text{ and }\quad A\coloneqq \bigcup_{t=1}^{\tau}A_t.
\end{align*}
Note that if for some $1\le t \le \tau$, we have $A_t=\varnothing$, then the process stabilises, that is, for every $t'\ge t$, we would have that $A_{t'}=\varnothing$ and $H_{t'}=H_t$. 

We will make use of the following observations:
\begin{enumerate}[(i)]
    \item The sets $\{A_1,\ldots,A_\tau\}$ are pairwise disjoint;
%    \item if $A_{\tau}\neq \varnothing$, then $\tau=\log d$; and,
    \item\label{i:third} If $A_{\tau}=\varnothing$, then for every $v\in V(G)\setminus A$, $d_{H}(v) = (1\pm \delta)dp$.
\end{enumerate}
The first one is apparent by construction. To see \ref{i:third} assume that $A_{\tau}=\varnothing$. Then, every $v\in \left(V(G) \setminus A\right) = \left(V(G)\setminus \bigcup_{t=1}^{\tau-1}A_t\right)=V(H_{\tau})$ satisfies $d_{H}(v)\ge (1-\delta_{\tau})dp= (1-\delta)dp$. On the other hand, since $v \in \left(V(G) \setminus A\right) \subseteq \left(V(G) \setminus A_1\right)$, it follows that $d_H(v) \le d_{G_p}(v) \le (1 + \delta_1) dp \leq (1+\delta)dp$. 

In particular, by \ref{i:third}, in order to prove Theorem \ref{th: induced}, it suffices to show that \textbf{whp}\footnote{With high probability, that is, with probability tending to one as $d$ tends to infinity.} 
\begin{align*}
    A_{\tau}=\varnothing\quad\text{ and }\quad|A|=o(2^d).    
\end{align*} 
Key to the proof is the following lemma, which gives a simple-to-analyse condition for when $v\in A_t$ for $t\in [2,\tau]\cap\mathbb{N}$.
\begin{lemma}\label{l: key lemma}
Let $t\in[2, \tau]\cap \mathbb{N}$. If $v\in A_t$, then there is a set $X$ of at least $\left(\frac{\delta dp}{2t\log d}\right)^{t-1}$ vertices at distance exactly $t-1$ from $v$, such for all $x \in X$,
\[
d_{G_p}(x) \neq (1\pm \delta_1)dp.
\]
\end{lemma}
\begin{proof}
We will use the following observation about the structure of $Q^d$, which is easy to verify:
\begin{equation}\label{e:cubestructure}
\begin{array}{c}
\text{If $k \leq d$ and $v$ and $w$ are at distance $k$ in $Q^d$, then $w$ has precisely $d-k$ neighbours} \\
\text{ at distance $k+1$ from $v$ and $k$ neighbours at distance $k-1$ from $v$.}
\end{array}
\end{equation}
%Nb. the second follows from the first together with regularity and bipartiteness.
The lemma will follow from iteratively applying the next fairly simple claim.
\begin{claim}\label{c: helping lemma}
Let $t\in [2,\tau]\cap \mathbb{N}$ and $k\in [1,t]\cap\mathbb{N}$, let $v\in V(G)$, and let $S\subseteq \left(A_t \cap N^k(v)\right)$. Then, there exists a set $X\subseteq \left(A_{t-1}\cap N^{k+1}(v)\right)$ with $|X|\ge |S|\frac{\delta dp}{2(k+1)\log d}$.
\end{claim}
\begin{proof}
Since each $s\in S$ is in $A_t$, $d_{H_{t}}(s)<(1-\delta_{t})dp$. On the other hand, since $A_t$ and $A_{t-1}$ are disjoint, each $s \in S$ is not in $A_{t-1}$ and consequently $d_{H_{t-1}}(s)\ge (1-\delta_{t-1})dp$. Thus, there are at least $(\delta_{t}-\delta_{t-1})dp=\frac{\delta dp}{\lfloor\log d\rfloor}\ge \frac{\delta dp}{\log d}$ neighbours of $s$ which are in $A_{t-1}$, let us denote them by $Y_s$. 

If we let $X_s \coloneqq \left(Y_s \cap N^{k+1}(v)\right) \subseteq \left(A_{t-1}\cap N^{k+1}(v)\right)$, then by \eqref{e:cubestructure}, $|X_s| \geq \frac{\delta dp}{\log d}-k$. On the other hand, since $X \coloneqq  \bigcup_{s\in S}X_s \subseteq  N^{k+1}(v)$ and $S \subseteq N^{k}(v)$, again by \eqref{e:cubestructure} each $x \in X$ lies in at most $k+1$ sets $X_s$. It follows by a simple double counting argument that
\[
|X|\ge \frac{|S|}{k+1}\left(\frac{\delta dp}{\log d}-k\right)\ge |S|\frac{\delta dp}{2(k+1)\log d},
\]
where we used that $k\le \tau=\lfloor\log d\rfloor\le \log d$ and $dp\ge \log^5d$.
%where we used that $\delta dp=\omega(k \log d)$ since $k\leq \tau \leq \log d$ and $dp \ge \log^5d$.
\end{proof}
To complete the proof of Lemma \ref{l: key lemma}, note that if $v\in A_t$, then $v\notin A_{t-1}$, and by the same argument as above, there exists $S_1\subseteq \left(N(v)\cap A_{t-1}\right)$ with $|S_1|\ge \frac{\delta dp}{2\log d}$. 

Iteratively applying Claim \ref{c: helping lemma} to the sets $S_i$, for $i\in [1, t-2]\cap \mathbb{N}$, we obtain a sequence of sets $S_{i+1}\subseteq \left(N^{i+1}(v)\cap A_{t-{i+1}}\right)$ with $|S_{i+1}|\ge |S_i|\frac{\delta dp}{2(i+1)\log d}$. It follows that $$X\coloneqq S_{t-1}\subseteq \left(N^{t-1}(v)\cap A_{1}\right)$$ has size at least
\[
\prod_{i=1}^{t-1}\frac{\delta dp}{2(i+1)\log d}=\frac{1}{t!}\left(\frac{\delta dp}{2\log d}\right)^{t-1}\ge \left(\frac{\delta dp}{2t\log d}\right)^{t-1},
\]
where we used $t!\le t^{t-1}$. Since $X \subseteq A_1$, by \eqref{e:A1} $X$ satisfies the assertion of the lemma. 
\end{proof}

With these lemmas at hand, we are now ready to prove Theorem \ref{th: induced}.
\begin{proof}[Proof of Theorem \ref{th: induced}]
It suffices to show that \textbf{whp} $A_{\tau}=\varnothing$ and $|A|=o(2^d)$.

Fix $v\in V(G)$ and $t\in [2,\tau]\cap\mathbb{N}$. We start by showing
\begin{equation}\label{e:pvinAt}
\mathbb{P}[ v \in A_t] \leq \exp \left\{ - \left(\frac{\delta dp}{2t\log d}\right)^{t-1}\right\}.
\end{equation}

Indeed, by Lemma \ref{l: key lemma}, if $v\in A_t$, then there is a set $X$ of at least $\left(\frac{\delta dp}{2t\log d}\right)^{t-1}$ vertices at distance $t-1$ from $v$, such that $d_{G_p}(x) \neq (1 \pm \delta_1)dp$ for all $x \in X$. Furthermore, since every $x\in X$ is at distance exactly $t-1$ from $v$, they have the same parity, and thus $X$ is an independent set in $G$. For each $x \in X$, $d_{G_p}(x) \sim \text{Bin}(d,p)$ and so by Lemma \ref{l: chernoff}, we have that
\begin{align*}
\mathbb{P}\left[d_{G_p}(x) \neq (1\pm \delta_1)dp) \right ]&\le 2\exp\left\{-\frac{\delta_1^2d^2p^2}{3dp}\right\}\le \exp\left\{-\frac{\delta^2dp}{4\log^2d}\right\}.
\end{align*}
%where we used that $dp=\omega(\log^2d)$ since $dp \ge \log^5 d$.
Since $X\subseteq N^{t-1}(v)$, there are at most $\binom{|N^{t-1}(v)|}{|X|}$ possible choices for $X$. Note that $|N^{t-1}(v)|=\binom{d}{t-1}$ (choosing $t-1$ of the $d$ coordinates to obtain a vertex at distance $t-1$ from $v$). Recalling that $d_{G_p}(x)$ is independent for each $x\in X$ and that $|X|=r\ge \left(\frac{\delta dp}{2t\log d}\right)^{t-1}$, by a union bound we obtain
\begin{align*}
    \mathbb{P}[v\in A_t]&\le \binom{\binom{d}{t-1}}{r}\exp\left\{-r\cdot \frac{\delta^2dp}{4\log^2d}\right\}\le \left(\left(ed\right)^{t-1}\exp\left\{-3\log^2d\right\}\right)^{r}\\
    &\le \exp\left\{r\left(2\log^2d - 3\log^2 d\right)\right\}\le \exp\left\{-2r\right\}\\
    &\leq \exp\left\{-\left(\frac{\delta dp}{2t\log d}\right)^{t-1}\right\},
\end{align*}
where the third inequality follows since $t\le \tau = \lfloor\log d\rfloor\le \log d$. Note that in this estimation, we used (very generously) our assumption of $p\ge \frac{\log^5d}{d}$, that is, that the numerator is polylogarithmic in $d$.
%where the second inequality follows since $dp \ge \log^5d$ and the third inequality follows since $t\le \log d$. 

We first note that \eqref{e:pvinAt} implies that for each $t \in [2,\tau]\cap \mathbb{N}$,
\[
\mathbb{E}\left[|A_t|\right] \le 2^d\exp\left\{-\left(\frac{\delta dp}{2t\log d}\right)^{t-1}\right\} \le 2^d\exp\left\{-\log^2d\right\}.
\]
Also, by a similar application of Lemma \ref{l: chernoff}, we have
$$\mathbb{E}[|A_1|]\le 2^d\exp\left\{-\frac{\delta^2dp}{4\log^2d}\right\}\le 2^d\exp\left\{-\log^2d\right\}.$$
Recalling that $A\coloneqq \bigcup_{t=1}^{\tau}A_t$, it follows that $\mathbb{E}[|A|]\le 2^d \tau \exp\left\{-\log^2d\right\} = o(2^d)$, and so by Markov's inequality, \textbf{whp} $|A|=o(2^d)$.

Secondly, by \eqref{e:pvinAt} we obtain
\begin{align*}
\mathbb{P}[ A_\tau \neq \emptyset] &\leq \sum_{v \in V(G)} \mathbb{P}[v \in A_\tau] \leq 2^d\exp\left\{-\left(\frac{\delta dp}{5\log^2 d}\right)^{\tau-1}\right\}\\
&\le 2^d\exp\left\{-\left(\log^2d\right)^{\lfloor\log d\rfloor-1}\right\}\le 2^d\exp\{-d\}=o(1).
\end{align*}
%where the third inequality follows since $dp\ge \log^5 d$.
\end{proof}
%\subsection{Finding a nearly perfect matching}
%Theorem \ref{th: perfect matching} follows rather immediately from Vizing's Theorem \cite{V64} and Theorem \ref{th: matching}.
%\begin{theorem}[Vizing's Theorem]
%Given a graph $H$ with maximum degree $\Delta(H)$, there exists a proper colouring of its edges with $\Delta(H)+1$ colours.
%\end{theorem}
%Indeed, by Theorem \ref{th: matching} \textbf{whp} there exists a subgraph $H\subseteq Q^d_p$ with $(1-o(1))2^d$ vertices such that $d_H(x) = (1\pm o(1))dp$ for all $x \in V(H)$. In particular, $\Delta(H)\le (1+o(1))dp$ and $|E(H)|\ge (1-o(1))2^{d-1}dp$. Therefore, by Vizing's theorem, there is a matching in $H$ (and thus in $Q^d_p$) covering at least $2\cdot \frac{|E(H)|}{\Delta(H)+1}\ge (1-o(1))2^d$ vertices. \qed
%\begin{align*}
%    2\cdot \frac{|E(H)|}{\Delta(H)+1}\ge (1-o(1))2^d
%\end{align*}
%vertices.%, where we used that $dp = \omega(1)$.
%\qed
\section{Discussion}\label{s: discussion}
We showed that for any $d$-regular graph $G$ on $n$ vertices, for every $\epsilon>0$ there exists a constant $C>0$ such that if $d\ge C$ and $p\ge \frac{C}{d}$, then there typically is a matching on at least $(1-\epsilon)n$ vertices in $G_p$. Further, we showed that when $p\ge \frac{\log^5}{d}$, $Q^d_p$ typically contains an induced nearly spanning subgraph, whose degrees are tightly concentrated around the expected degree $dp$. To find this subgraph, we employed a fairly simple pruning process.

In recent years there has been an interest in the universality of properties of $G(n,p)$, and in particular in extending results on the quantitative similarity of the structure of $G(n,p)$ and $Q^d_p$ to broader classes of \emph{high-dimensional} graphs. For example, the typical emergence of the giant component, its uniqueness and its asymptotic properties have been considered \cite{CDE24,DEKK22, DEKK23, DEKK24, L22}. Moreover, Diskin and Geisler \cite{DG24} extended the result of Bollob\'as \cite{B90} to these settings as well, roughly showing that for any $d$-regular high-dimensional product graph, the hitting times of minimum degree one, connectivity, and perfect matching are \textbf{whp} the same. 
It is not hard to verify that the proof laid out in this paper generalises, almost verbatim, to $d$-regular high-dimensional Cartesian product graphs (specifically, the $t$-dimensional product of regular graphs of bounded order). In fact, one can even further relax the assumptions.
\begin{theorem}\label{th: general}
Let $G$ be a $d$-regular graph with $d=\omega(1)$. Suppose that for every $v\in V(G)$, every $k \leq\log d$, and every $u\in N^k(v)$, we have that $\left|N(u)\cap \bigcup_{i=1}^kN^i(v)\right|=O(\log d)$. Let $p\ge \frac{\log^5d}{d}$. Then, \textbf{whp} $G_p$ contains an induced subgraph $H$ such that $|V(H)|=(1-o_d(1))|V(G)|$ and for every $v \in V(H)$, $d_H(v) = (1\pm o_d(1))dp$. 
\end{theorem}
%Just to remind myself - |G| plays no role in the analysis so don't need a bound there. The bipartiteness is a red herring, the backwards expansion property means that the degrees of vertices at a fixed distance from x are only very slightly correlated.
This raises the following more general question. 
\begin{question}
Let $G$ be a $d$-regular graph with $d=\omega(1)$. What `minimal' assumptions on $G$ and $p$ suffice to have: \textbf{whp} $G_p$ has an induced nearly spanning, nearly regular subgraph with all degrees concentrated around $dp$?
\end{question}

As an application of Theorem \ref{th: matching} we have that when $p\ge \frac{C}{d}$, \textbf{whp} a largest matching in $Q^d_p$ contains $(1-o_C(1))2^{d}$ vertices, that is, the first order term is $2^d$ whereas the second order term is bounded by $o_C(2^d)$. Much more precise results are known in the setting of $G(n,p)$~\cite{F86}. It would be interesting to determine more precisely the typical size of a largest matching in $Q^d_p$ when $p=\frac{C}{d}$ (for $d>C$). A first step to answer this would be to determine the second order term --- there, it is perhaps natural to conjecture that the size of the `defect set' (that is, the set of vertices that are not covered by a largest matching in $Q^d_p$) is dominated by the number of isolated vertices and thus \textbf{whp} a largest matching typically covers $2^d-(2(1-p))^d+o_C((2(1-p))^d)$ vertices (indeed, this would resonate with the known hitting time results \cite{B90}).

\paragraph{Acknowledgement} The second and third authors were supported in part by the Austrian Science Fund (FWF) [10.55776/\text{P36131}, 10.55776/\text{F1002}]. Part of this work was done while the first and the fourth authors were visiting Graz University of Technology, and they would like to thank the Graz University of Technology for its hospitality. For the purpose of open access, the authors have applied a CC-BY public copyright licence to any Author Accepted Manuscript version arising from this submission.

\bibliographystyle{abbrv}
\bibliography{perc}
\end{document}